\documentclass[12pt,a4paper,oneside]{amsart}
\usepackage{fullpage}
\usepackage{amssymb,amsfonts}
\usepackage{epsfig}
\usepackage{graphicx}

\usepackage{graphics}
\newtheorem{theorem}{Theorem}[section]
\newtheorem{lemma}[theorem]{Lemma}

\theoremstyle{definition}
\newtheorem{definition}[theorem]{Definition}
\newtheorem{remark}[theorem]{Remark}

\numberwithin{equation}{section}

\DeclareMathOperator{\diam}{diam} 

\newcommand{\be}{\begin{equation}}
\newcommand{\ee}{\end{equation}}

\DeclareMathOperator{\loc}{loc}

\makeindex

\def\Xint#1{\mathchoice 
 {\XXint\displaystyle\textstyle{#1}}%
{\XXint\textstyle\scriptstyle{#1}}%
{\XXint\scriptstyle\scriptscriptstyle{#1}}%
 {\XXint\scriptscriptstyle\scriptscriptstyle{#1}}%
 \!\int}
\def\XXint#1#2#3{{\setbox0=\hbox{$#1{#2#3}{\int}$}
 \vcenter{\hbox{$#2#3$}}\kern-.5\wd0}}

 \def\dashint{\Xint-}

\begin{document}
\title{Removable sets for weighted  Orlicz-Sobolev spaces}
\author{Nijjwal Karak}
\address{Indian Statistical Institute Chennai Centre,  SETS campus, MGR Knowledge City, CIT Campus, Taramani, Chennai, 600 113}
\email{nijjwal@gmail.com}
\begin{abstract}
The aim in the present paper is to study removable sets for weighted Orlicz-Sobolev spaces. We generalize the definition of porous sets and show that the porous sets lying in a hyperplane are removable.
\end{abstract}
\maketitle
\section{Introduction}
Let $\mathbb{R}^n$ $(n\geq 2)$ denote the $n$-dimensional Euclidean space and $\Omega$ be an open set in $\mathbb{R}^n.$ We will often write a point $x\in\mathbb{R}^n$ as $x=(x',x_n),$ where $x'=(x_1,x_2,\ldots,x_{n-1})$ and $x_n\in\mathbb{R}^1.$\\
\indent In \cite{Kos99}, Koskela studied removable sets for Sobolev spaces $W^{1,p}(\mathbb{R}^n).$ If $E\subset\mathbb{R}^n$ is a closed set of zero Lebesgue $n$-measure, then we say that $E$ is removable for $W^{1,p}(\mathbb{R}^n)$ if $W^{1,p}(\mathbb{R}^n\setminus E)=W^{1,p}(\mathbb{R}^n)$ as sets. Recall that $u\in W^{1,p}(\Omega),$ $1\leq p<\infty$ provided $u\in L^p(\Omega)$ and there are functions $\partial_j u\in L^p(\Omega),$ $j=1,2,\ldots,n,$ so that
\begin{equation}\label{weak}
\int_{\Omega}u\partial_j\phi\,dx=-\int_{\Omega}\phi\partial_ju\,dx
\end{equation}
for each test function $\phi\in C_0^1(\Omega)$ and all $1\leq j\leq n.$ Here $dx$ denotes the usual Lebesgue measure. It is immediate that removability is a local behavior, that is, $E$ is removable for $W^{1,p}(\mathbb{R}^n)$ if and only if for each $x\in E$ there exists $r>0$ so that $W^{1,p}(B(x,r)\setminus E)=W^{1,p}(B(x,r))$ as sets. Also observe that $E$ is removable if and only if for each $u\in W^{1,p}(\mathbb{R}^n\setminus E)$ the functions $\partial_j u\in L^p(\mathbb{R}^n\setminus E)$ satisfy \eqref{weak} (with $\Omega=\mathbb{R}^n$) for each $\phi\in C_0^1(\mathbb{R}^n)$ and not only for $\phi\in C_0^1(\mathbb{R}^n\setminus E).$\\
\indent The structure of removable sets has been studied by several authors. Ahlfors and Beurling \cite{AB50} introduced the so called NED sets as the sets whose removal does not effect extremal length and proved that those sets are the removable singularities for Dirichlet finite analytic and univalent functions. The connection between removable sets and removable singularities for quasiconformal mappings has been studied in \cite{Res89}, \cite{KW96} and \cite{Wu98}. However, Koskela's work in \cite{Kos99} produce concrete criteria for removability using a concept of $p$-porosity.\\
\indent Porosity of a set lying on a hyperplane is essentially determined by their thickness, see \cite{Kos99} for the concrete definition. In \cite{Kos99}, Koskela proved that if $E$ is $p$-porous, then $E$ is removable for $W^{1,p}(\mathbb{R}^n).$ This result has been extended to Orlicz-Sobolev spaces in \cite{Kar15} and to weighted Sobolev spaces in \cite{FM03}. In this paper, we consider this problem of removability for weighted Orlicz-Sobolev spaces.\\
\indent Let $p>1,$ $-1<\alpha<p-1$ and $\lambda\in\mathbb{R}.$ When $E$ is restricted to subsets of $\mathbb{R}^{n-1},$ we consider the weights of the form $\rho(x)^{\alpha},$ where $\rho(x)$ denotes the distance of $x\in\mathbb{R}^n$ from the hyperplane $\mathbb{R}^{n-1},$ that is, $\rho(x)=\vert x_n\vert$ for $x=(x',x_n)\in\mathbb{R}^n.$ Let $\mu_{\alpha}$ be the Borel measure
\begin{equation*}
d\mu_{\alpha}(x)=\rho(x)^{\alpha}dx=\vert x_n\vert^{\alpha}dx,
\end{equation*}
where $x=(x',x_n)\in\mathbb{R}^n.$ Let $W^{1,\Psi}(\Omega,\mu_{\alpha})$ denote the weighted Orlicz-Sobolev spaces of all functions $u\in L^{\Psi}(\Omega,\mu_{\alpha})$ whose distributional gradient $\nabla u=(\partial_1u,\ldots,\partial_nu)$ also belongs to $L^{\Psi}(\Omega,\mu_{\alpha}),$ where $\Psi(t)=t^p\log^{\lambda}(e+t).$ Weighted Sobolev spaces has been studied in the book \cite{HKM06}, whereas details on Orlicz spaces and Orlicz-Sobolev spaces can be found in the book \cite{RR91} and in the dissertation \cite{Tuo04} respectively. If $E\subset\mathbb{R}^n$ is a closed set with $\mu_{\alpha}(E)=0,$ then we say that $E$ is removable for $W^{1,\Psi}(\mathbb{R}^n,\mu_{\alpha})$ if
\begin{equation*}
W^{1,\Psi}(\mathbb{R}^n\setminus E,\mu_{\alpha})=W^{1,\Psi}(\mathbb{R}^n,\mu_{\alpha}).
\end{equation*}    
As in \cite{Kos99}, $E$ is removable for $W^{1,\Psi}(\mathbb{R}^n,\mu_{\alpha})$ if and only if each function $u\in W^{1,\Psi}(\mathbb{R}^n\setminus E,\mu_{\alpha})$ the functions $\partial_ju\in L^{\Psi}(\mathbb{R}^n\setminus E,\mu_{\alpha})$ satisfy \eqref{weak} (with $\Omega=\mathbb{R}^n$) for each $\phi\in C_0^1(\mathbb{R}^n)$ and not only for $\phi\in C_0^1(\mathbb{R}^n\setminus E).$\\
\indent In \cite{Kar15}, the author has defined $(p,\lambda)$-porosity and proved that $(p,\lambda)$-porous sets lying on the hyperplane $\mathbb{R}^{n-1}$ are removable for $W^{1,\Psi}(\mathbb{R}^n)$ for Orlicz functions $\Psi(t)=t^p\log^{\lambda}(e+t).$ In this paper, $(p,\lambda)$-porosity has been generalized to $(p,\lambda,\alpha)$-porosity and removabilty of these sets lying in the hyperplane $\mathbb{R}^{n-1}$ for $W^{1,\Psi}(\mathbb{R}^n,\mu_{\alpha})$ has been established. For the definition of $(p,\lambda,\alpha)$-porous sets, please see Section 3. Roughly speaking, a set $E$ is $(p,\lambda,\alpha)$-porous if every small neighbourhood of every point in $E$ contains a hole with certain diameter. These holes are either balls or continua depending on the parameters $p,\lambda$ and $\alpha.$ Here is our main theorem:\\
\textbf{Theorem A.} Let $-1<\alpha<p-1,$ $1<p<n+\alpha$ and $\lambda\in\mathbb{R}.$ If $E$ is $(p,\lambda,\alpha)$-porous, then $E$ is removable for $W^{1,\Psi}(\mathbb{R}^n,\mu_{\alpha}).$ This is also true for $p=n+\alpha,$ and $\lambda\leq n+\alpha-1,$ where $\alpha>-1.$\\
The restriction $\lambda<n+\alpha-1$ for $p=n+\alpha$ is natural beacause of the same reason as discussed in Section 3 of \cite{Kar15}.\\
\indent The main idea behind the removability of $(p,\lambda,\alpha)$-porous sets is the following. As mentioned above, it suffices to prove that \eqref{weak} holds for each $u\in C^1(\Omega\setminus E)\cap W^{1,\Psi}(\Omega\setminus E,\mu_{\alpha})$ and for each $\phi\in C_0^1(\Omega).$ By the Fubini theorem and the usual integration by parts it suffices to show that the one sided limits $\lim_{t\rightarrow 0+}u(x',t)$ and $\lim_{t\rightarrow 0-}u(x',t)$ coincide for $\mathcal{H}^{n-1}$-a.e. $x=(x',0)\in E.$ This is established via sharp capacity estimates and the existence of \lq\lq holes\rq\rq\ in $E$ guaranteed by the porosity condition. Some of the ideas have been borrowed from \cite[Theorem 5.9]{HK98} and \cite{Kar15}.\\ 
\indent {\textit{Acknowledgement}.} I wish to thank Professor Pekka Koskela for helpful comments and suggestions. 
\section{Notations and Preliminaries}
A function $\Psi:[0,\infty)\rightarrow [0,\infty)$ is a Young function if
\begin{align*}
\Psi(s)=\int_0^s \psi(t)\, dt,
\end{align*}
where $\psi:[0,\infty)\rightarrow [0,\infty)$ with $\psi(0)=0,$ is an increasing, left-continuous function which is neither identically zero nor identically infinite on $(0,\infty).$ A Young function $\Psi$ is convex, increasing, left-continuous and satisfies
\begin{align*}
\Psi(0)=0, \ \lim_{t\rightarrow\infty}\Psi(t)=\infty . 
\end{align*}
The generalized inverse of a Young function $\Psi,$ $\Psi^{-1}:[0,\infty]\rightarrow[0,\infty],$ is defined by the formula
\begin{equation*}
\Psi^{-1}(t)=\inf\{s:\Psi(s)>t\},
\end{equation*}
where $\inf(\emptyset)=\infty.$ A Young function $\Psi$ and its generalized inverse satisfy the double inequality
\begin{equation*}
\Psi(\Psi^{-1}(t))\leq t \leq \Psi^{-1}(\Psi(t))
\end{equation*}
for all $t\geq 0.$ In this article we will only consider the Young functions $\Psi(t)=t^p\log^{\lambda}(e+t)$ and in that case we have for all $t\geq 0$
\begin{equation}\label{inverse}
\Psi^{-1}(t)\approx t^{\frac{1}{p}}/\log^{\frac{\lambda}{p}}(e+t).
\end{equation}
Here $A\approx B$ denotes the two sided inequality $C_1B\leq A\leq C_2B$ for some constants $C_1$ and $C_2.$ For a general Young function $\Psi,$ a Borel measure $\mu$ and an open set $\Omega\subset\mathbb{R}^n,$ the Orlicz space $L^{\Psi}(\Omega,\mu)$ is defined by
\begin{align*}
L^{\Psi}(\Omega,\mu)=\{u:\Omega\rightarrow[-\infty,\infty]: u ~\text{measurable},~ \int_{\Omega}\Psi(\alpha\vert u\vert)\, d\mu<\infty ~\text{for some}~ \alpha>0\}.
\end{align*}
A function $u\in L^{\Psi}(\Omega,\mu)$ is in the Orlicz-Sobolev space $W^{1,\Psi}(\Omega,\mu)$ if its weak partial derivatives (distributional derivatives) $\partial_j u$ belong to $L^{\Psi}(\Omega,\mu)$ for all $1\leq j\leq n.$ In this article we write $L^{\Psi}(\Omega)$ and $W^{1,\Psi}(\Omega)$ instead of $L^{\Psi}(\Omega,\mathcal{L}^n)$ and $W^{1,\Psi}(\Omega,\mathcal{L}^n)$ respectively.\\
\indent Let us also recall the Poincar\'e inequality. A pair $u\in L_{\loc}^1(\Omega)$ and a measurable function $g\geq 0$ satisfy a $(1,p)$-Poincar\'e inequality, $p\geq 1,$ if there exist constants $C_p>0$ and $\tau\geq 1,$ such that
\begin{align}\label{poincare}
\dashint_B\vert u-u_B\vert\, dz\leq C_pr\left(\dashint_{\tau B}g^p\, dz\right)^\frac{1}{p}
\end{align}
for each ball $B=B(x,r)$ satisfying $\tau B\subset\Omega.$ Recall that if $\Omega\subset\mathbb{R}^n$ and $u\in W_{\loc}^{1,1}(\Omega),$ then the inequality \eqref{poincare} holds for $g=\vert \nabla u\vert$ with $\tau=1,$ $p=1$ and the constant depending only on $n.$ Here and throughout the article, $u_B$ is the average of $u$ in $B(x,r)$ and the barred integrals are the averaged integrals, that is $\dashint_A v\, d\mu=\mu(A)^{-1}\int_A v\, d\mu.$\\
\indent Let $A\subset\mathbb{R}^n,$ $0\leq \alpha<\infty$ and $0<\delta\leq\infty.$ We define $\alpha$-dimensional Hausdorff measure by setting
\begin{equation*}
\mathcal{H}^{\alpha}(E)=\limsup_{\delta\rightarrow 0}H_{\delta}^{\alpha}(E),
\end{equation*}
where
\begin{equation*}
H_{\delta}^{\alpha}(E)=\inf\sum_i c(\alpha)(\diam(B_i))^{\alpha},
\end{equation*}
where $c(\alpha)$ is a fixed constant and the infimum is taken over all collections of balls $\{B_i\}_{i=1}^{\infty}$ such that $\diam(B_i)\leq\delta$ and $E\subset\bigcup_{i=1}^{\infty} B_i.$ Recall that the \textit{Hausdorff $\alpha$-content} of a set $E\subset\mathbb{R}^n$ is the number
$$\mathcal{H}_{\infty}^{\alpha}(E)=\inf\sum_i(\diam(B_i))^{\alpha},$$
where the infimum is taken over all countable covers of the set $E$ by balls $B_i.$\\
\indent For the convenience of reader we state here a fundamental covering lemma (for a proof see \cite[2.8.4-6]{Fed69} or \cite[Theorem 1.3.1]{Zie89}).
\begin{lemma}[5B-covering lemma]\label{cover}
Every family $\mathcal{F}$ of balls of uniformly bounded diameter in a metric space $X$ contains a pairwise disjoint subfamily $\mathcal{G}$ such that for every $B\in\mathcal{F}$ there exists $B'\in\mathcal{G}$ with $B\cap B'\neq\emptyset$ and $\diam(B)<2\diam(B').$ In particular, we have that
$$\bigcup_{B\in\mathcal{F}}B\subset\bigcup_{B\in\mathcal{G}}5B.$$
\end{lemma}
It is also worthy to mention that we use $C, c$ and $M$ to denote constants and they may change in every line, whereas we use $C_x$ and $c_x$ to denote constants depending on the point $x.$
\section{Proof of Theorem A}
First we generalize the definition of $(p,\lambda)$-porosity in \cite{Kar15} to $(p,\lambda,\alpha)$-porosity in a suitable way.
\begin{definition}\label{porosity}
Let $p>1, -1<\alpha<p-1$ and $\lambda\in\mathbb{R}.$ We say that $E\subset\mathbb{R}^{n-1}$ is $(p,\lambda,\alpha)$-porous, if for $\mathcal{H}^{n-1}$-a.e. $x\in E,$ there is a sequence of $r_i>0$ and a constant $c_x>0$ such that $r_i\rightarrow 0$ as $i\rightarrow\infty$ and each $(n-1)$-dimensional ball $B(x,r_i)$ contains\\
(i) a ball $B_i\subset B(x,r_i)\setminus E$ of radius $R_i$ with $R_i^{n+\alpha-p}\log^{\lambda}\left(\frac{1}{R_i}\right)\geq C_x{r_i}^{n-1}$ when $1+\alpha<p<n+\alpha-1$ and $\lambda\in\mathbb{R},$\\
(ii) a ball $B_i\subset B(x,r_i)\setminus E$ of radius $R_i$ with $R_i\log^{\lambda}\left(\frac{1}{R_i}\right)\geq C_x{r_i}^{n-1}$ when $p=n+\alpha-1$ and $\lambda\leq n+\alpha-1,$\\
(iii) a continuum $F_i\subset B(x,r_i)\setminus E$ of diameter $R_i$ with $R_i\log^{\lambda-(n+\alpha-1)}\left(\frac{1}{R_i}\right)\geq C_x{r_i}^{n-1}$ when $p=n+\alpha-1$ and $\lambda>n+\alpha-1,$\\
(iv) a continuum $F_i\subset B(x,r_i)\setminus E$ of diameter $R_i$ with $R_i^{n+\alpha-p}\log^{\lambda}\left(\frac{1}{R_i}\right)\geq C_x{r_i}^{n-1}$ when $n+\alpha-1<p\leq n+\alpha$ and $\lambda\in\mathbb{R},$
\end{definition}
\noindent Note that $(p,\lambda,0)$-porosity is same as $(p,\lambda)$-porosity defined in \cite{Kar15} except for the cases $p=n-1,$ $\lambda\leq n-1$ and $p=n.$\\
The following lemma is the key to prove Theorem A.
\begin{lemma}\label{mainlemma}
Let $\xi\in\mathbb{R}^{n-1}$ and $r>0.$ 
Let $A\subset B^n(\xi,r)\cap\mathbb{R}^{n-1}$ be a continuum with $\mathcal{H}^1_{\infty}(A)\geq r/3,$ when $n+\alpha-1<p\leq n+\alpha,$ $\lambda\in\mathbb{R}$ or $p=n+\alpha-1,$ $\lambda>n+\alpha-1.$ Otherwise let $A\subset B^n(\xi,r)\cap\mathbb{R}^{n-1}$ be a set with $\mathcal{H}^{n-1}_{\infty}(A)\geq r/3,$ when $1+\alpha<p<n+\alpha-1,$ $\lambda\in\mathbb{R}$ or $p=n+\alpha-1,$ $\lambda\leq n+\alpha-1.$ Suppose that $u\in W^{1,\Psi}(B^n(\xi,r)^+)\cap C(B^n(\xi,r)^+\cup A)$ satisfies $u=0$ on $A$ and $u_{B^n(\xi,r)^+}\geq 1/2,$ where $B^n(\xi,r)^+$ denotes the upper half of the $n$-dimensional ball $B(\xi,r).$ Then
\begin{equation*}
\int_{B^n(\xi,2r)^+}\Psi(\vert\nabla u\vert)\,d\mu_{\alpha}\geq
\begin{cases}
cr\log^{\lambda-(n+\alpha-1)}\left(\frac{1}{r}\right)& \text{when}~~ p=n+\alpha-1, \lambda>n+\alpha-1,\\
cr^{n+\alpha-p}\log^{\lambda}\left(\frac{1}{r}\right)& \text{otherwise}.
\end{cases}
\end{equation*}
\end{lemma}
\begin{proof}
Suppose that there exists a point $y\in A$ such that $u_{B^n(y,r)^+}\leq 1/3.$ Then we have by Poincar\'e inequality,
\begin{eqnarray}\label{lessthanzero}
\frac{1}{6}\leq\vert u_{B^n(y,r)^+}-u_{B^n(\xi,r)^+}\vert &\leq &M\dashint_{B^n(\xi,2r)^+}\vert u-u_{B^n(\xi,2r)^+}\vert\,d\mu_{\alpha}\nonumber\\
&\leq &Mr\dashint_{B^n(\xi,2r)^+}\vert\nabla u\vert\,d\mu_{\alpha}\\
&\leq &Mr\left(\dashint_{B^n(\xi,2r)^+}\vert\nabla u\vert ^p\,d\mu_{\alpha}\right)^{1/p}\nonumber\\
&\leq &Mr^{1-\frac{n+\alpha}{p}}\left(\int_{B^n(\xi,2r)^+}\vert\nabla u\vert ^p\,d\mu_{\alpha}\right)^{1/p}.\nonumber
\end{eqnarray}
First we consider the case $\lambda\geq 0.$ For this we split $B^n(\xi,2r)^+$ into \lq\lq good\rq\rq\, part $B^n(\xi,2r)^g$ and \lq\lq bad\rq\rq\, part $B^n(\xi,2r)^b$ where $B^n(\xi,2r)^g=\{x\in B^n(\xi,2r)^+:\vert\nabla u(x)\vert\leq r^{-1/2}\}$ and $B^n(\xi,2r)^b=\{x\in B^n(\xi,2r)^+:\vert\nabla u(x)\vert> r^{-1/2}\}.$ Using this splitting we obtain
\begin{equation*}
\frac{1}{6}\leq Mr^{1/2}+Mr^{1-\frac{n+\alpha}{p}}\log^{-\frac{\lambda}{p}}(e+r^{-1/2})\left(\int_{B^n(\xi,2r)^b}\vert\nabla u\vert ^p\log^{\lambda}(e+\vert\nabla u\vert)\,d\mu_{\alpha}\right)^{\frac{1}{p}}
\end{equation*}
which implies that
\begin{equation*}
\int_{B^n(\xi,2r)^+}\Psi(\vert\nabla u\vert)\,d\mu_{\alpha}\geq Cr^{n+\alpha-p}\log^{\lambda}\left(\frac{1}{r}\right)\left(\frac{1}{6}-Mr^{1/2}\right).
\end{equation*}
For the case $\lambda<0,$ we apply Jensen's inequality to \eqref{lessthanzero} and use \eqref{inverse} to get
\begin{eqnarray*}
\frac{1}{6}&\leq & Mr\Psi^{-1}\left(\dashint_{B^n(\xi,2r)^+}\Psi(\vert\nabla u\vert)\,d\mu_{\alpha}\right)\\
&\leq &\frac{Mr\left(\dashint_{B^n(\xi,2r)^+}\Psi(\vert\nabla u\vert)\,d\mu_{\alpha}\right)^{1/p}}{\log^{\frac{\lambda}{p}}\left(e+\dashint_{B^n(\xi,2r)^+}\Psi(\vert\nabla u\vert)\,d\mu_{\alpha}\right)}
\end{eqnarray*}
We can choose a constant $M'$ such that $\int_{B^n(\xi,2r)^+}\Psi(\vert\nabla u\vert)\,d\mu_{\alpha}\leq M'.$ Therefore
\begin{equation*}
\int_{B^n(\xi,2r)^+}\Psi(\vert\nabla u\vert)\,d\mu_{\alpha}\geq Mr^{n+\alpha-p}\log^{\lambda}\left(e+M'r^{-(n+\alpha)}\right).
\end{equation*} 
Now we may assume that $u_{B^n(x,r)^+}\geq 1/3$ for all $x\in A.$ Since every point of $A$ is a Lebesgue point of $u,$ we have by Poincar\'e inequality and H\"older inequality,
\begin{eqnarray}\label{lessthanzero2}
\frac{1}{3} \leq \vert u(x)-u_{B^n(x,r)^+}\vert\nonumber &\leq &\sum_{j=0}^{\infty}\vert u_{B^n(x,2^{-j}r)^+}-u_{B^n(x,2^{-j-1}r)^+}\vert\nonumber\\
&\leq & c\sum_{j=0}^{\infty}\dashint_{B^n(x,2^{-j}r)^+}\vert u-u_{B^n(x,2^{-j}r)^+}\vert\,d\mu_{\alpha}\nonumber\\
&\leq & c\sum_{j=0}^{\infty}2^{-j}r\dashint_{B^n(x,2^{-j}r)^+}\vert\nabla u\vert\,d\mu_{\alpha}\\
&\leq & c\sum_{j=0}^{\infty}2^{-j}r\left(\dashint_{B^n(x,2^{-j}r)^+}\vert\nabla u\vert ^p\,d\mu_{\alpha}\right)^{\frac{1}{p}}.\nonumber
\end{eqnarray}
We again split $B^n(x,2^{-j}r)^+$ into \lq\lq good\rq\rq\, parts $B^n(x,2^{-j}r)^g=\{z\in B^n(x,2^{-j}r)^+:\vert\nabla u(z)\vert\leq (2^{-j}r)^{-1/2}\}$ and \lq\lq bad\rq\rq\, parts $B^n(x,2^{-j}r)^b=\{z\in B^n(x,2^{-j}r)^+:\vert\nabla u(z)\vert>(2^{-j}r)^{-1/2}\}$ for all $j.$ This gives
\begin{equation*}
\frac{1}{3}\leq c\sum_{j=0}^{\infty}(2^{-j}r)^{\frac{1}{2}}+c\sum_{j=0}^{\infty}(2^{-j}r)^{1-\frac{n+\alpha}{p}}\log^{-\frac{\lambda}{p}}\left(e+(2^{-j}r)^{-\frac{1}{2}}\right)\left(\int_{B^n(x,2^{-j}r)^b}\Psi(\vert\nabla u\vert)\,d\mu_{\alpha}\right)^{\frac{1}{p}}
\end{equation*}
and hence
\begin{equation}\label{cases}
\frac{1}{3}-2cr^{\frac{1}{2}}\leq c\sum_{j=0}^{\infty}(2^{-j}r)^{1-\frac{n+\alpha}{p}}\log^{-\frac{\lambda}{p}}\left(e+(2^{-j}r)^{-\frac{1}{2}}\right)\left(\int_{B^n(x,2^{-j}r)^b}\Psi(\vert\nabla u\vert)\,d\mu_{\alpha}\right)^{\frac{1}{p}},
\end{equation}
for $\lambda\geq 0$ (for $\lambda<0,$ we apply Jensen's inequality to \eqref{lessthanzero2} and use \eqref{inverse} as above to get such an inequality).\\
\textbf{Case I. $n+\alpha-1<p\leq n+\alpha,$ $\lambda\in\mathbb{R}.$} If for each $j=0,1,2,\ldots$ and for a fixed $\epsilon>0$
\begin{equation*}
\int_{B^n(x,2^{-j}r)^b}\Psi(\vert\nabla u\vert)\,d\mu_{\alpha}\leq\epsilon r^{n+\alpha-1-p}(2^{-j}r)\log^{\lambda}\left(\frac{1}{r}\right),
\end{equation*}
then we have that
\begin{equation*}
\frac{1}{3}-2cr^{\frac{1}{2}}\leq c\epsilon^{1/p}\log^{\frac{\lambda}{p}}\left(\frac{1}{r}\right)\sum_{j=0}^{\infty}(2^{-j})^{(p-n-\alpha+1)/p}\log^{-\frac{\lambda}{p}}\left(e+(2^{-j}r)^{-\frac{1}{2}}\right)\leq c\epsilon^{1/p},
\end{equation*}
because $p>n-1+\alpha.$ It follows that there is an index $j_x$ such that
\begin{equation*}
\int_{B^n(x,2^{-j_x}r)^+}\Psi(\vert\nabla u\vert)\,d\mu_{\alpha}\geq\epsilon_0 r^{n+\alpha-1-p}(2^{-j_x}r)\log^{\lambda}\left(\frac{1}{r}\right)
\end{equation*}
for some $\epsilon_0>0.$ Using 5B-covering lemma we find a pairwise disjoint collection of balls of the form $B_k=B^n(x_k,r_kr)$ such that $A\subset \cup_k5B_k$ and
\begin{equation*}
r_kr\leq Cr^{p+1-n-\alpha}\log^{-\lambda}\left(\frac{1}{r}\right)\int_{B_k^+}\Psi(\vert\nabla u\vert)\,d\mu_{\alpha}.
\end{equation*}
Hence we have
\begin{eqnarray*}
\frac{r}{3}\leq\mathcal{H}^1_{\infty}(A)\leq\sum_{k}5r_kr &\leq & Cr^{p+1-n-\alpha}\log^{-\lambda}\left(\frac{1}{r}\right)\sum_{k}\int_{B_k^+}\Psi(\vert\nabla u\vert)\,d\mu_{\alpha}\\
&\leq & Cr^{p+1-n-\alpha}\log^{-\lambda}\left(\frac{1}{r}\right)\int_{B^n(\xi,2r)^+}\Psi(\vert\nabla u\vert)\,d\mu_{\alpha},
\end{eqnarray*}
as desired.\\
Note that for $n=2$ it is enough to consider only Case I, but for $n\geq 3$ we will consider the following cases too.\\
\textbf{Case II. $p=n+\alpha-1,$ $\lambda>n+\alpha-1.$} We substitute the value of $p$ in \eqref{cases} to obtain
\begin{equation*}
\frac{1}{3}-2cr^{\frac{1}{2}}\leq c\sum_{j=0}^{\infty}(2^{-j}r)^{-\frac{1}{n+\alpha-1}}\log^{-\frac{\lambda}{n+\alpha-1}}\left(e+(2^{-j}r)^{-\frac{1}{2}}\right)\left(\int_{B^n(x,2^{-j}r)^b}\Psi(\vert\nabla u\vert)\,d\mu_{\alpha}\right)^{\frac{1}{n+\alpha-1}}.
\end{equation*}
If for each $j=0,1,2,\ldots$ and for a fixed $\epsilon>0$
\begin{equation*}
\int_{B^n(x,2^{-j}r)^b}\Psi(\vert\nabla u\vert)\,d\mu_{\alpha}\leq\epsilon (2^{-j}r)\log^{\lambda-(n+\alpha-1)}\left(\frac{1}{r}\right),
\end{equation*}
then we have that
\begin{equation*}
\frac{1}{3}-2cr^{\frac{1}{2}}\leq c\epsilon^{\frac{1}{n+\alpha-1}}\log^{\frac{\lambda}{n+\alpha-1}-1}\left(\frac{1}{r}\right)\sum_{j=0}^{\infty}\log^{-\frac{\lambda}{n+\alpha-1}}\left(e+(2^{-j}r)^{-\frac{1}{2}}\right)\leq c\epsilon^{\frac{1}{n+\alpha-1}}.
\end{equation*}
This guarantees us the existence of an index $j_x$ such that
\begin{equation*}
\int_{B^n(x,2^{-j_x}r)^+}\Psi(\vert\nabla u\vert)\,d\mu_{\alpha}\geq\epsilon_0 (2^{-j_x}r)\log^{\lambda-(n+\alpha-1)}\left(\frac{1}{r}\right)
\end{equation*}
for some $\epsilon_0>0.$ Then we use $5B$-covering lemma as above and the fact that $\frac{r}{3}\leq\mathcal{H}^1_{\infty}(A)$ to obtain
\begin{equation*}
\int_{B^n(\xi,2r)^+}\Psi(\vert\nabla u\vert)\,d\mu_{\alpha}\geq cr\log^{\lambda-(n+\alpha-1)}\left(\frac{1}{r}\right).
\end{equation*}
\textbf{Case III. $1+\alpha<p<n+\alpha-1$ $\lambda\in\mathbb{R}.$} If for each $j=0,1,2,\ldots$ and for a fixed $\epsilon>0$
\begin{equation*}
\int_{B^n(x,2^{-j}r)^b}\Psi(\vert\nabla u\vert)\,d\mu_{\alpha}\leq\epsilon r^{\alpha+1-p}(2^{-j}r)^{n-1}\log^{\lambda}\left(\frac{1}{r}\right),
\end{equation*}
then from \eqref{cases} we have that
\begin{equation*}
\frac{1}{3}-2cr^{\frac{1}{2}}\leq c\epsilon^{1/p}\log^{\frac{\lambda}{p}}\left(\frac{1}{r}\right)\sum_{j=0}^{\infty}(2^{-j})^{(p-\alpha-1)/p}\log^{-\frac{\lambda}{p}}\left(e+(2^{-j}r)^{-\frac{1}{2}}\right)\leq c\epsilon^{1/p},
\end{equation*}
because $\alpha<p-1.$ Keeping in mind that $\mathcal{H}^{n-1}_{\infty}(A)\geq r^{n-1}/3,$ we can prove, using 5B-covering lemma similar to the previous cases, that
\begin{equation*}
\int_{B^n(\xi,2r)^+}\Psi(\vert\nabla u\vert)\,d\mu_{\alpha}\geq cr^{n+\alpha-p}\log^{\lambda}\left(\frac{1}{r}\right).
\end{equation*}
\textbf{Case IV. $p=n+\alpha-1,$ $\lambda\leq n+\alpha-1.$} If for each $j=0,1,2,\ldots$ and for a fixed $\epsilon>0$
\begin{equation*}
\int_{B^n(x,2^{-j}r)^b}\Psi(\vert\nabla u\vert)\,d\mu_{\alpha}\leq\epsilon(2^{-j}r)^{n-1}\log^{\lambda}\left(\frac{1}{r}\right),
\end{equation*}
then from \eqref{cases} we have that
\begin{equation*}
\frac{1}{3}-2cr^{\frac{1}{2}}\leq c\epsilon^{\frac{1}{n+\alpha-1}}\log^{\frac{\lambda}{n+\alpha-1}}\left(\frac{1}{r}\right)\sum_{j=0}^{\infty}(2^{-j})^{\frac{n-2}{n+\alpha-1}}\log^{-\frac{\lambda}{n+\alpha-1}}\left(e+(2^{-j}r)^{-\frac{1}{2}}\right)\leq c\epsilon^{\frac{1}{n+\alpha-1}},
\end{equation*}
for $n>2.$ Since $\mathcal{H}^{n-1}_{\infty}(A)\geq r^{n-1}/3,$ the same procedure yields us
\begin{equation*}
\int_{B^n(\xi,2r)^+}\Psi(\vert\nabla u\vert)\,d\mu_{\alpha}\geq cr\log^{\lambda}\left(\frac{1}{r}\right).
\end{equation*}
\end{proof}
\noindent\textbf{Proof of Theorem A.}
For simplicity we can assume that $E\subset I_{n-1}=(0,1)^{n-1}.$ It is easy to check that $(p,\lambda,\alpha)$-removability is equivalent to the requirement that for each $u\in W^{1,\Psi}(B^n(0,2)\setminus E,\mu_{\alpha})\cap C^1(B^n(0,2)\setminus E),$ $u^+(x)=u^-(x)$ holds for $\mathcal{H}^{n-1}$-a.e. $x\in E.$ Here $u^+(x)=\lim_{t\rightarrow 0+}u(x_1,x_2,\ldots,x_{n-1},t),$ $u^-(x)=\lim_{t\rightarrow 0-}u(x_1,x_2,\ldots,x_{n-1},t)$ and these limits exist for $\mathcal{H}^{n-1}$-a.e. $x=(x_1,x_2,\ldots,x_{n-1},0)\in E,$ by the Fubini theorem and the fundamental theorem of calculus. Let $u\in W^{1,\Psi}(B^n(0,2)\setminus E,\mu_{\alpha})\cap C^1(B^n(0,2)\setminus E).$ Therefore, it is enough to show that
\begin{equation}\label{decay}
\int_{B^n(x,r_i)}\Psi(\vert \nabla u(z)\vert)\,d\mu_{\alpha}(z)\geq C_xr_i^{n-1}
\end{equation} 
for all large $i$ whenever $x=(x_1,x_2,\ldots,x_{n-1},0)\in E$ is such that the one-sided limits $u^+(x),$ $u^-(x)$ do not coincide and the porosity condition holds at $x.$ Here $B^n(x,r_i)$ is the $n$-dimensional ball corresponding to the $(n-1)$-dimensional ball $B(x,r_i)$ from the definition of porosity. Let $B_i, F_i$ and $R_i$ also be retained from the definition of porosity and take $\zeta_i\in \mathbb{R}^{n-1}\cap B^n(x,r_i)$ are the centres of the balls $B_i$ in Definition \ref{porosity} but take $\zeta_i\in \mathbb{R}^{n-1}\cap B^n(x,r_i)$ such that $F_i\subset B^n(\zeta_i,R_i)$ when we have continua in Definition \ref{porosity}. By symmetry and porosity we may assume that the upper limit is $1,$ the lower limit is $0$ and $u\leq 1/2$ in a set $A\subset B_i$ with $\mathcal{H}^{n-1}_{\infty}(A)\geq\frac{1}{2}\mathcal{H}^{n-1}_{\infty}(B_i)$ or in a compact set $A\subset F_i$ with $\mathcal{H}^{1}_{\infty}(A)\geq\frac{1}{2}\mathcal{H}^{1}_{\infty}(F_i).$\\
Let $I_n=(0,1)^n$ and $I_n^+$ denotes the upper half of $I_n.$ For any $\epsilon>0,$ we can write $u\in W^{1,1+\alpha+\epsilon}(I_n^+,\mu_{\alpha})$ when $\alpha\geq 0,$ otherwise we can write $u\in W^{1,1+\epsilon}(I_n^+,\mu_{\alpha}).$ In any case, we use H\"older inequality to deduce that $u\in W^{1,1}(I_n^+).$  Then we reflect $u$ with respect to the hyperplane $x_n=0$ to obtain a function $v\in W^{1,1}(I_n),$ which coincides with $u$ in the upper half plane. Now since $u$ coincides with $v$ in the upper half plane and our weight $\vert x_n\vert^{\alpha}$ is symmetric with respect to the hyperplane we reflect, so we get $v\in W^{1,1+\alpha+\epsilon}(I_n,\mu_{\alpha})$ when $\alpha\geq 0$ or $v\in W^{1,1+\epsilon}(I_n,\mu_{\alpha})$ otherwise. Applying Theorem 3 of \cite{FM03} for $v$ and using Remarks $1$ \& $2$ of it, we conclude that for $\mathcal{H}^{n-1}$-a.e. $x\in\mathbb{R}^{n-1},$ $\dashint_{B^n(z,z_n/2)}v(y)\,d\mu_{\alpha}(y)$ tends to $1$ as $z\rightarrow x$ along $T(x,a)$ for every $a>0.$ Here $T(x,a)=\{z\in\mathbb{R}^n_+:\vert z-x\vert <az_n\},$ where $\mathbb{R}^n_+=\{x=(x',x_n):x_n>0\},$ and $a>0.$ Now, assuming that $i$ is large enough, we can take a ball $B^n(y_i,R_i)$ with $y_i=\zeta_i+(0,0,\ldots,2R_i)$ such that $u_{B^n(y_i,R_i)^+}\geq 1/2.$\\
\indent Suppose that $\dashint_{B^n(\zeta_i,R_i)^+}u\,d\mu_{\alpha}\leq 1/3.$ Then we have by Poincar\'e inequality
\begin{eqnarray*}
\frac{1}{6}\leq\vert u_{B^n(\zeta_i, R_i)^+}-u_{B^n(y_i,R_i)^+}\vert &\leq & C\dashint_{B^n(\zeta_i, 3R_i)^+}\vert u-u_{B^n(\zeta_i, 3R_i)^+}\vert\,d\mu_{\alpha}\\
&\leq & CR_i\left(\dashint_{B^n(\zeta_i, 3R_i)^+}\vert \nabla u\vert ^p\,d\mu_{\alpha}\right)^{1/p}.
\end{eqnarray*}
Then we split the balls into \lq\lq good\rq\rq\ parts and \lq\lq bad\rq\rq\ parts as in Lemma \ref{mainlemma} to obtain
\begin{equation}\label{equation1}
\int_{B^n(x,7r_i)}\Psi(\vert \nabla u\vert)\,d\mu_{\alpha}\geq\int_{B^n(\zeta_i,3R_i)^+}\Psi(\vert \nabla u\vert)\,d\mu_{\alpha}\geq CR_i^{n+\alpha-p}\log^{\lambda}\left(\frac{1}{R_i}\right)
\end{equation}
for $\lambda\geq 0$ (for $\lambda<0$ we apply Jensen's inequality and use \eqref{inverse} similarly as in Lemma \ref{mainlemma} to obtain this estimate).\\
\indent Suppose then that $\dashint_{B^n(\zeta_i,R_i)^+}u\,d\mu_{\alpha}\geq 1/3.$ We apply Lemma \ref{mainlemma} to conclude that
\begin{equation}\label{equation2}
\int_{B^n(x,r_i)}\Psi(\vert\nabla u\vert)\, d\mu_{\alpha}\geq
\begin{cases}
CR_i\log^{\lambda-(n+\alpha-1)}\left(\frac{1}{R_i}\right)&~\text{when}~p=n+\alpha-1,~\lambda>n+\alpha-1,\\
CR_i^{n+\alpha-p}\log^{\lambda}\left(\frac{1}{R_i}\right)&~\text{otherwise}.
\end{cases}
\end{equation}
We take respective minimums of the two inequalities \eqref{equation1} and \eqref{equation2} and use the definition of porosity to complete the proof.
\begin{remark}
As in \cite{Kar15} and \cite{Kos99}, we can also construct a $(p,\lambda,\alpha)$-porous set which is not removable for $W^{1,\Psi}(\mathbb{R}^n,\mu_{\alpha})$ for $\Psi'(t)=t^p\log^{\lambda-\epsilon}(e+t)$ for any $\epsilon>0.$  
\end{remark}
\def\bibname{References}
\bibliography{removable_weighted}
\bibliographystyle{alpha}
\end{document}